\documentclass{amsart}

\title[]{Strong convergence theorems for strongly relatively
nonexpansive sequences and applications}

\author{Koji~Aoyama}
\address[Koji Aoyama]{%
Department of Economics, Chiba University, 
Yayoi-cho, Inage-ku, Chi\-ba-shi, Chiba 263-8522, Japan}
\email{aoyama@le.chiba-u.ac.jp}

\author{Yasunori~Kimura}
\address[Yasunori~Kimura]{%
Department of Information Science, Toho University, Miyama, Funabashi-shi,
Chiba 274-8510, Japan}
\email{yasunori@is.sci.toho-u.ac.jp}

\author{Fumiaki~Kohsaka}
\address[Fumiaki Kohsaka]{%
Department of Computer Science and Intelligent Systems, 
Oita University, Dannoharu, Oita-shi, Oita 870-1192, Japan}
\email{f-kohsaka@oita-u.ac.jp}

\date{\today}

\usepackage{amssymb}
\usepackage[initials]{amsrefs}

\def\N{\mathbb N}
\providecommand{\norm}[1]{\left\lVert#1\right\rVert}
\providecommand{\ip}[2]{\left\langle #1, #2 \right\rangle}
\providecommand{\dom}{\operatorname{dom}}

\allowdisplaybreaks
\numberwithin{equation}{section}

\theoremstyle{plain}
\newtheorem{theorem}{Theorem}[section]
\newtheorem{lemma}[theorem]{Lemma}

\theoremstyle{definition}
\newtheorem{example}[theorem]{Example}

\theoremstyle{remark}
\newtheorem{remark}[theorem]{Remark}

\begin{document}

\begin{abstract}
 The aim of this paper is to establish 
 strong convergence theorems for a strongly relatively nonexpansive
 sequence in a smooth and uniformly convex Banach space. 
 Then we employ our results to approximate solutions of 
 the zero point problem for a maximal monotone operator 
 and the fixed point problem for a relatively nonexpansive mapping. 
\end{abstract}

\maketitle

\section{Introduction}
Let $E$ be a smooth and uniformly convex Banach space, 
$E^*$ the dual of $E$, $A \subset E\times E^*$ a maximal monotone
operator with a zero point, 
and $\{r_n\}$ a sequence of positive real numbers. 
Assume that $\{x_n\}$ is a sequence defined as follows: $x_1 \in E$ and 
\[
  x_{n+1} = J^{-1} \left(
  \frac1n J x + \Bigl( 1 - \frac1n \Bigr) J (J+ r_n A)^{-1}J x_n \right)
\]
for $n \in \N$, 
where $J$ and $J^{-1}$ are the duality mappings of $E$ and $E^*$, 
respectively. 
It is known~\cite{MR2058504} that 
if $r_n \to \infty$, then $\{x_n\}$ converges strongly to some zero
point of $A$. 
However, we have not known whether 
$\{x_n\}$ converges strongly or not 
without the assumption that $r_n \to \infty$. 
In \S\ref{s:app} we present an affirmative answer to this problem; 
see Theorem~\ref{th:mm} and Remark~\ref{r:ak_kt}. 

Furthermore, more general results are proved; see Theorem~\ref{th:tauSRNS}. 
This is a strong convergence theorem for a strongly relatively
nonexpansive sequence introduced in~\cite{MR2529497}. 
In the proofs of Theorem~\ref{th:tauSRNS}, we use modifications of ideas
developed in~\cites{MR2466027,MR2680036}.
In particular, Lemma~\ref{lm:mainge} due to Maing\'e~\cite{MR2466027} is
a fundamental tool; see also Example~\ref{ex} and
Lemma~\ref{lm:neo_mainge}.

In \S\ref{s:app}, 
using Theorem~\ref{th:tauSRNS}, we also show Theorem~\ref{th:fpp}
which is a strong convergence theorem for a relatively nonexpansive
mapping in the sense of Matsushita and Takahashi~\cite{MR2058234}.

\section{Preliminaries}
Throughout the present paper, 
$E$ denotes a real Banach space with norm $\norm{\,\cdot\,}$, 
$E^*$ the dual of $E$, 
$\ip{x}{x^*}$ the value of $x^*\in E^*$ at $x\in E$, 
and $\N$ the set of positive integers. 
The norm of $E^*$ is also denoted by $\norm{\,\cdot\,}$. 
Strong convergence of a sequence $\{x_n\}$ in $E$ to $x\in E$ is denoted
by $x_n \to x$ and weak convergence by $x_n \rightharpoonup x$. 
The (normalized) duality mapping of $E$ is denoted by $J$, that is, 
it is a set-valued mapping of $E$ into $E^*$ defined by 
\[
Jx = \{ x^*\in E^* : \ip{x}{x^*} = \norm{x}^2 = \norm{x^*}^2 \}
\]
for $x\in E$. 

Let $S_E$ denote the unit sphere of $E$, that is, 
$S_E =\{ x\in E : \norm{x}=1 \}$.
The norm $\norm{\,\cdot\,}$ of $E$ is said to be G\^{a}teaux
differentiable if the limit
\begin{equation}
\lim_{t \to 0} \frac{\norm{x+ty} -\norm{x}}t 
 \label{eqn:norm-diff}
\end{equation}
exists for all $x,y \in S_E$.
In this case $E$ is said to be smooth
and it is known that the duality mapping $J$ of $E$ is single-valued. 
The norm of $E$ is said to be uniformly G\^{a}teaux differentiable
if for each $y\in S_E$ the limit \eqref{eqn:norm-diff} is attained
uniformly for $x \in S_E$.
A Banach space $E$ is said to be uniformly smooth
if the limit~\eqref{eqn:norm-diff} is attained uniformly for $x,y \in S_E$. 
In this case it is known that $J$ is uniformly norm-to-norm continuous 
on each bounded subset of $E$; see \cite{MR1864294} for more details.

A Banach space $E$ is said to be strictly convex
if $x,y\in S_E$ and $x \ne y$ imply $\norm{x+y}<2$. 
A Banach space $E$ is said to be uniformly convex 
if for any $\epsilon >0$ there exists $\delta >0$ such that 
$x,y\in S_E$ and $\norm{x-y}\geq \epsilon$ imply 
$\norm{x+y}/2 \leq 1-\delta$.
It is known that 
$E$ is reflexive and strictly convex if $E$ is uniformly convex; 
$E$ is uniformly smooth if and only if $E^*$ is uniformly convex;
see~\cite{MR1864294} for more details. 

In the rest of this section, unless otherwise stated, 
we assume that $E$ is a smooth, strictly convex, and reflexive Banach
space. 
In this case it is known that 
the duality mapping $J$ of $E$ is single-valued and bijective,
and $J^{-1}$ is the duality mapping of $E^*$. 

We deal with a real-valued function $\phi$ on $E\times E$ defined by
\[
 \phi(x,y)=\norm{x}^2 - 2\ip{x}{Jy} + \norm{y}^2
\]
for $x,y\in E$; see \cites{MR1386667,MR1972223}. 
From the definition of $\phi$, it is clear that
\begin{equation}\label{eqn:phi-lower}
 (\norm{x}- \norm{y})^2 \leq \phi(x,y) 
\end{equation}
for all $x,y\in E$. 
Since $\norm{\,\cdot\,}^2$ is convex, 
\begin{equation}\label{eq:J-convex}
 \phi \Bigl( w,J^{-1} \bigl( \lambda Jx + (1-\lambda)Jy \bigr) \Bigr)
  \leq \lambda \phi(w,x) + (1-\lambda)\phi(w,y)
\end{equation}
holds for all $x,y,w\in E$ and $\lambda \in [0,1]$. 
It is known that 
\begin{equation}\label{eq:KT2004}
 \phi(x,J^{-1}x^*) \leq \phi(x, J^{-1}(x^* - y^*)) 
  + 2 \ip{J^{-1}x^* -x}{y^*}
\end{equation}
holds for all $x\in E$ and $x^*, y^*\in E^*$; 
see~\cite{MR2058504}*{Lemma~3.2}. 

\begin{lemma}[\cite{MR1972223}*{Proposition 2}]\label{lm:kamimura2002}
 Let $E$ be a smooth and uniformly convex Banach space. 
 Let $\{x_n\}$ and $\{y_n\}$ be bounded sequences in $E$. 
 If $\phi(x_n,y_n)\to 0$, then $x_n - y_n \to 0$. 
\end{lemma}

Let $\{x_n\}$ and $\{y_n\}$ be bounded sequences in $E$. 
Then it is obvious from the definition of $\phi$ that 
$\phi(x_n,y_n)\to 0$ if $x_n - y_n\to 0$. 
From this fact and Lemma~\ref{lm:kamimura2002}, we deduce the following: 
If $E$ is a uniformly convex and uniformly smooth Banach space $E$, then
\begin{equation}\label{eq:equiv_x_n-y_n}
 x_n - y_n \to 0 \Leftrightarrow Jx_n - Jy_n \to 0 
  \Leftrightarrow \phi(x_n, y_n) \to 0.
\end{equation}

In the rest of this section, we assume that $C$ is a nonempty closed
convex subset of $E$. 

Let $T\colon C\to E$ be a mapping. 
The set of fixed points of $T$ is denoted by $F(T)$. 
A point $p\in C$ is said to be an asymptotic fixed point of $T$ 
\cites{MR1386686,MR1402641} if there exists a sequence $\{x_n\}$ in $C$ 
such that $x_n \rightharpoonup p$ and $x_n -T x_n\to 0$. 
The set of asymptotic fixed points of $T$ is denoted by
$\hat{F}(T)$. 
A mapping $T$ is said to be of type~(r)
if $F(T) \ne \emptyset$ and $\phi(p,Tx)\leq \phi(p,x)$ 
for all $x\in C$ and $p\in F(T)$; 
$T$ is said to be relatively nonexpansive~\cites{MR2058234,MR2142300}
if $T$ is of type~(r) and $F(T)= \hat{F}(T)$. 
We know that if $T\colon C\to E$ is of type~(r), then $F(T)$ is closed
and convex; see~\cite{MR2142300}*{Proposition~2.4}. 

It is known that, for each $x \in E$, 
there exists a unique point $x_0 \in C$ such that 
\[
 \phi(x_0,x)=\min\{\phi(y,x): y\in C\}. 
\]
Such a point $x_0$ is denoted by $Q_C (x)$ and $Q_C$ is called the
generalized projection of $E$ onto $C$; see \cites{MR1386667,MR1972223}. 
It is known that 
\begin{equation}\label{eqn:generalized-projection}
 \ip{z - Q_C (x)}{Jx - JQ_C(x)} \le 0 
\end{equation}
or equivalently
\begin{equation}\label{eq:gp_by_phi}
 \phi(z,Q_C(x)) + \phi(Q_C(x),x) \le \phi(z,x)
\end{equation}
holds for all $x\in E$ and $z\in C$. 
It is obvious from~\eqref{eq:gp_by_phi} that 
the generalized projection $Q_C$ is of type~(r). 

Let $A$ be a set-valued mapping of $E$ into $E^*$, 
which is denoted by $A\subset E\times E^*$. 
The effective domain of $A$ is denoted by $\dom(A)$ and the range of $A$
by $R(A)$, that is, $\dom(A)=\{x\in E: Ax \ne \emptyset\}$ and 
$R(A)=\bigcup_{x\in \dom(A)}Ax$.
A set-valued mapping $A\subset E\times E^*$ is said to be a monotone
operator if $\ip{x-y}{x^* - y^*}\geq 0$
for all $(x,x^*), (y,y^*)\in A$. 
A monotone operator $A\subset E\times E^*$ is said to be maximal 
if $A=A'$ whenever $A'\subset E\times E^*$ is a monotone operator such
that $A\subset A'$.
It is known that if $A$ is a maximal monotone operator, then $A^{-1} 0$
is closed and convex, where $A^{-1}0 =\{x\in E: Ax \ni 0\}$. 

Let $A\subset E\times E^*$ be a maximal monotone operator and $r>0$. 
Then it is known that $R(J + rA) = E^*$; see~\cites{MR0282272}. 
Thus a single-valued mapping $L_r = (J+rA)^{-1}J$ of $E$ onto $\dom(A)$ 
is well defined and is called the resolvent of $A$. 
It is also known that $F(L_r) = A^{-1}0$ and 
\begin{equation}
 \label{eqn:resolvent}
 \phi(u, L_r x) + \phi(L_r x, x) \leq \phi(u,x) 
\end{equation}
for all $x \in E$ and $u\in F(L_r)$; see~\cites{MR2058504,MR2112848}.
It is obvious from~\eqref{eqn:resolvent} that 
the resolvent $L_r$ of $A$ is of type~(r) for all $r>0$ 
whenever $A^{-1}0$ is nonempty.

The following lemma is well known; see~\cites{MR1911872,MR2338104}. 

\begin{lemma}\label{lm:seq}
 Let $\{ \xi_n \}$ be a sequence of nonnegative real numbers, 
 $\{ \gamma_n \}$ a sequence of real numbers, 
 and $\{ \alpha_n \}$ a sequence in $[0,1]$. 
 Suppose that 
 $\xi_{n+1} \leq (1- \alpha_n) \xi_n + \alpha_n \gamma_n$
 for every $n\in\N$, $\limsup_{n\to\infty} \gamma_n \leq 0 $, 
 and $\sum_{n=1}^\infty \alpha_n = \infty$. Then $\xi_n \to 0$. 
\end{lemma}

\section{Eventually increasing functions and strongly relatively
 nonexpansive sequences}

In this section, we provide some needed lemmas about an eventually
increasing function and a strongly relatively nonexpansive sequence. 

A function $\tau \colon \N \to \N$ is said to be \emph{eventually
increasing} if $\lim_{n\to\infty}\tau(n) = \infty$ and 
$\tau(n) \leq \tau(n+1)$ for all $n\in \N$. 
By definition, we easily obtain the following: 

\begin{lemma}\label{2zero}
 Let $\tau \colon \N \to \N$ be an eventually increasing function and
 $\{\alpha_n\}$ a sequence of real numbers such that $\alpha_n \to 0$. 
 Then $\alpha_{\tau(n)}\to 0$. 
\end{lemma}

We need the following lemma: 

\begin{lemma}[Maing{\'e}~\cite{MR2466027}*{Lemma 3.1}]\label{lm:mainge}
 Let $\{\xi_n\}$ be a sequence of real numbers. 
 Suppose that there exists a subsequence $\{\xi_{n_i}\}$ of $\{\xi_n\}$ 
 such that $\xi_{n_i} < \xi_{n_i + 1}$ for all $i\in \N$. 
 Then there exist $N\in \N$ and a function $\tau \colon \N \to \N$ such
 that $\tau(n) \leq \tau(n+1)$, $\xi_{\tau(n)} \le \xi_{\tau(n)+1}$, and 
 $\xi_n \le \xi_{\tau(n)+1}$ for all $n \geq N$
 and $\lim_{n\to \infty} \tau(n) = \infty$. 
\end{lemma}

Under the assumptions of Lemma~\ref{lm:mainge}, we can not choose a
strictly increasing function $\tau$; see the following example: 

\begin{example}\label{ex}
 Let $\{ \xi_n \}$ be a sequence of real numbers define by 
 \[
 \xi_n = \begin{cases}
	  0 & \text{if $n$ is odd;}\\ 
	  1/n & \text{if $n$ is even.}
	 \end{cases}
 \]
 Then the following hold:
 \begin{enumerate}
  \item There exists a subsequence $\{ \xi_{n_i} \}$ of $\{ \xi_n \}$
	such that $\xi_{n_i} < \xi_{n_{i}+1}$ for all $i\in \N$;
  \item there does not exist a subsequence $\{ \xi_{m_k} \}$ of 
	$\{ \xi_n \}$ such that  
	$\xi_{m_k} \leq \xi_{m_k +1}$ and
	$\xi_k \leq \xi_{m_k +1}$ for all $k\in \N$. 
 \end{enumerate}
\end{example}

\begin{proof}
 Define $n_i = 2i-1$ for each $i\in\N$. Then it is clear that 
 \[
 \xi_{n_i} = \xi_{2i-1} = 0 < \frac1{2i} = \xi_{2i} = \xi_{n_i +1}
 \]
 for every $i\in\N$. Thus (1) holds. 

 Let $\{ \xi_{m_k} \}$ be a subsequence of $\{ \xi_n \}$. 
 Suppose that $\xi_{m_k} \leq \xi_{m_k +1}$ for all $k \in \N$. 
 Then it is easy to check that $m_k$ is odd and $m_k+1$ is even 
 for every $k \in \N$. 
 We now assume that $\xi_k \leq \xi_{m_k +1}$ for $k \in \N$. 
 Then it follows that 
  \[
  \frac1k = \xi_k \leq \xi_{m_k +1} = \frac1{m_k +1}
  \]
 if $k$ is even. 
 This implies that $k \geq m_k +1 \geq k +1$, which is a contradiction. 
\end{proof}

Using Lemma~\ref{lm:mainge}, we obtain the following: 

\begin{lemma}\label{lm:neo_mainge}
 Let $\{\xi_n\}$ be a sequence of nonnegative real numbers which is not
 convergent. 
 Then there exist $N\in \N$ and an eventually increasing function 
 $\tau \colon \N \to \N$ such that
 $\xi_{\tau(n)} \le \xi_{\tau(n)+1}$ for all $n \in \N$ and 
 $\xi_n \le \xi_{\tau(n)+1}$ for all $n \geq N$. 
\end{lemma}

\begin{proof}
 Since $\{\xi_n\}$ is not convergent, 
 for any $n\in \N$ there exists $m \in \N$ such that 
 $m \geq n$ and $\xi_{m} < \xi_{m+1}$, 
 and hence there exists a subsequence $\{\xi_{n_i}\}$ of $\{\xi_n\}$ 
 such that $\xi_{n_i} < \xi_{n_i + 1}$ for every $i\in \N$. 
 Lemma~\ref{lm:mainge} implies that there exist $N\in \N$ 
 and a function $\sigma \colon \N \to \N$ such that 
 $\sigma(n) \leq \sigma(n+1)$, $\xi_{\sigma(n)} \le \xi_{\sigma(n)+1}$,
 and $\xi_n \le \xi_{\sigma(n)+1}$ for every $n \geq N$
 and $\lim_{n\to \infty} \sigma(n) = \infty$.
 Let us define $\tau\colon \N \to \N$ by 
 $\tau(n) = \sigma(N)$ for $n\in \{1, 2, \dotsc, N \}$ 
 and $\tau(n) = \sigma(n)$ for $n > N$, 
 which completes the proof. 
\end{proof}

In the rest of this section, unless otherwise stated, 
we assume that $E$ is a smooth, strictly convex, and reflexive Banach
space and $C$ is a nonempty closed convex subset of $E$. 

Let $\{T_n\}$ be a sequence of mappings of $C$ into $E$ 
such that $F=\bigcap_{n=1}^\infty F(T_n)$ is nonempty. Then 
\begin{itemize}
 \item $\{T_n\}$ is said to be a \emph{strongly relatively nonexpansive
       sequence}~\cite{MR2529497} if each $T_n$ is of type~(r) and 
       $\phi(T_n x_n,x_n)\to 0$ whenever $\{x_n\}$ is a bounded sequence
       in $E$ and $\phi(p,x_n)-\phi(p,T_nx_n)\to 0$ for some point $p\in
       F$; 
 \item $\{T_n\}$ satisfies the \emph{condition~(Z)} if every weak
       cluster point of $\{x_n\}$ belongs to $F$ whenever $\{ x_n \}$ is
       a bounded sequence in $C$ such that $T_n x_n - x_n\to 0$. 
\end{itemize}
Let $A \subset E\times E^*$ be a maximal monotone operator with a zero
point and $\{r_n\}$ a sequence of positive real numbers. 
Then \eqref{eqn:resolvent} shows that the sequence $\{L_{r_n}\}$ of
resolvents of $A$ is a strongly relatively nonexpansive sequence;
see~\cite{MR2529497} for more details. 

In order to prove our main result in \S\ref{s:srns}, 
we need the following lemmas: 

\begin{lemma}\label{lm:tauSRNS}
 Let $\{T_n\}$ be a sequence of mappings of $C$ into $E$ such that 
 $F=\bigcap_{n=1}^\infty F(T_n)$ is nonempty, 
 $\tau\colon \N \to \N$ an eventually increasing function, 
 and $\{z_n\}$ a bounded sequence in $C$ such that 
 $\phi(p,z_n) - \phi(p, T_{\tau(n)} z_n) \to 0$ for some $p \in F$. 
 If $\{T_n\}$ is a strongly relatively nonexpansive sequence, 
 then $\phi(T_{\tau(n)} z_n, z_n) \to 0$. 
\end{lemma}

\begin{proof}
 Suppose that $\phi(T_{\tau(n)} z_n, z_n) \nrightarrow 0$. Then there
 exist $\epsilon > 0$  and a strictly increasing function 
 $\sigma \colon \N \to \N$ such that
 $\tau \circ \sigma$ is also strictly increasing and
 \begin{equation}\label{geq_e}
  \phi(T_{\tau \circ \sigma (n)} z_{\sigma (n)}, z_{\sigma (n)}) 
   \geq \epsilon
 \end{equation}
 for all $n \in \N$. 
 Set $\mu = \tau \circ \sigma$ and $R(\mu) = \{\mu(n): n \in \N\}$.  
 Define a sequence $\{y_n\}$ in $C$ as follows:
 For each $n\in \N$, 
 \[
 y_n = \begin{cases}
	z_{\sigma \circ \mu^{-1}(n)} & \text{if $n \in R (\mu)$};\\
	p & \text{if $n \notin R (\mu)$}. 
       \end{cases}
 \]
 It is clear that $\{ y_n \}$ is bounded, 
 \[
 \phi(p,y_n) - \phi(p, T_n y_n) 
 = \phi(p,z_{\sigma \circ \mu^{-1}(n)}) - 
 \phi(p, T_{\tau (\sigma \circ \mu^{-1}(n))} z_{\sigma \circ\mu^{-1}(n)})
 \]
 for $n \in R (\mu)$, 
 and $\phi(p,y_n) - \phi(p, T_n y_n) = 0$ for $n \notin R (\mu)$. 
 Since $\sigma \circ \mu^{-1}$ is strictly increasing, 
 it follows that $\phi(p,y_n) - \phi(p, T_n y_n) \to 0$, 
 so $\phi(T_n y_n, y_n) \to 0$ 
 because $\{T_n\}$ is a strongly relatively nonexpansive sequence. 
 Therefore, noting that $y_{\mu(n)} = z_{\sigma(\mu^{-1}(\mu(n)))} =
 z_{\sigma(n)}$ and $\mu$ is strictly increasing, we have 
 \[
 \phi(T_{\tau\circ\sigma (n)} z_{\sigma(n)},  z_{\sigma(n)}) 
 = 
 \phi(T_{\mu (n)} y_{\mu (n)}, 
 y_{\mu (n)}) \to 0, 
 \]
 which contradicts to \eqref{geq_e}. 
\end{proof}


\begin{lemma}\label{lm:tauZ}
 Let $\{T_n\}$ be a sequence of mappings of $C$ into $E$ such that 
 $F=\bigcap_{n=1}^\infty F(T_n)$ is nonempty, 
 $\tau\colon \N \to \N$ an eventually increasing function, and $\{z_n\}$ a
 bounded sequence in $C$ such that $T_{\tau(n)} z_n - z_n \to 0$. 
 Suppose that $\{T_n\}$ satisfies the condition~(Z). 
 Then every weak cluster point of $\{z_n\}$ belongs to $F$. 
\end{lemma}

\begin{proof}
 Let $z$ be a weak cluster point of $\{z_n\}$. 
 Then there exists a strictly increasing function $\sigma \colon \N \to \N$
 such that $z_{\sigma(n)} \rightharpoonup z$ as $n\to \infty$ and 
 $\tau \circ \sigma$ is strictly increasing. 
 Set $\mu = \tau \circ \sigma$ and $R(\mu) = \{\mu(n): n \in \N\}$. 
 Define a sequence $\{y_n\}$ in $C$ as follows:
 For each $n\in \N$, 
 \[
  y_n = \begin{cases}
	 z_{\sigma \circ \mu^{-1}(n)} & \text{if } n \in  R(\mu);\\
	 p & \text{if } n \notin R (\mu),
	\end{cases}
 \]
 where $p$ is a point in $F$. 
 Then it is clear that $\{ y_n \}$ is bounded, 
 \[
 y_n - T_n y_n = z_{\sigma \circ \mu^{-1}(n)} 
 -  T_{\tau (\sigma \circ \mu^{-1}(n))} z_{\sigma \circ \mu^{-1}(n)}
 \]
 for $n \in R (\mu)$, and $y_n - T_n y_n = 0$ for $n \notin R (\mu)$. 
 Since $z_n - T_{\tau(n)} z_n \to 0$ and $\sigma \circ \mu^{-1}$ is
 strictly increasing, it follows that $y_n - T_n y_n \to 0$. 
 Noting that $\mu$ is strictly increasing and 
 $y_{\mu(n)} = z_{\sigma \circ \mu^{-1} (\mu (n))} = z_{\sigma(n)}$
 for every $n\in \N$, 
 we know that $\{z_{\sigma (n)}\}$ is a subsequence of $\{y_n\}$, 
 and hence $z$ is a weak cluster point of $\{ y_n \}$. 
 Since $\{T_n\}$ satisfies the condition~(Z), we conclude that $z \in F$. 
\end{proof}

\begin{lemma}\label{lm:limsup_le_0}
 Let $\{T_n\}$ be a sequence of mappings of $C$ into $E$, 
 $F$ be a nonempty closed convex subset of $E$, 
 $\{z_n\}$ a bounded sequence in $C$ such that $z_n - T_n z_n \to 0$, 
 and $u \in E$. 
 Suppose that every weak cluster point of $\{z_n\}$ belongs to $F$. Then 
 \[
 \limsup_{n\to \infty} \ip{T_n z_n - w}{Ju - Jw} \leq 0 , 
 \]
 where $w=Q_F(u)$. 
\end{lemma}

\begin{proof}
 Since $z_n - T_n z_n \to 0$ and $\{z_n\}$ is bounded, 
 there exists a weakly convergent subsequence 
 $\{ z_{n_i} \}$ of $\{ z_n \}$ such that 
 \begin{align*}
  \limsup_{n\to \infty} \ip{T_n z_n - w}{Ju-Jw}
  &= \limsup_{n\to \infty} \ip{z_n - w}{Ju-Jw}\\
  &= \lim_{i\to \infty} \ip{z_{n_i} - w}{Ju-Jw}.
 \end{align*}
 Let $z$ be the weak limit of $\{ z_{n_i} \}$. 
 By assumption, we see that $z \in F$. 
 Thus~\eqref{eqn:generalized-projection} shows that 
 \[
 \lim_{i\to \infty} \ip{z_{n_i} - w}{Ju-Jw} = \ip{z - w}{Ju-Jw} \leq 0, 
 \]
 which is the desired result.
\end{proof}

\section{Strong convergence theorems for strongly relatively
 nonexpansive sequences}\label{s:srns}

In this section, we prove the following strong convergence theorem: 

\begin{theorem}\label{th:tauSRNS}
 Let $E$ be a smooth and uniformly convex Banach space, 
 $C$ a nonempty closed convex subset of $E$, 
 $\{S_n\}$ a sequence of mappings of $C$ into $E$
 such that $F = \bigcap_{n=1}^\infty F(S_n)$ is nonempty,
 and $\{\alpha_n\}$ a sequence in $[0,1]$ such that $\alpha_n \to 0$. 
 Let $u$ be a point in $E$ and $\{x_n\}$ a sequence defined by 
 $x_1 \in C$ and 
 \begin{equation}\label{eq:def1}
  x_{n+1}= Q_C J^{-1}\bigl( \alpha_n J u + (1-\alpha_n)J S_n x_n \bigr)
 \end{equation}
 for $n\in\N$. Suppose that 
 \begin{itemize}
  \item $\{S_n\}$ is a strongly relatively nonexpansive sequence; 
  \item $\{S_n\}$ satisfies the condition~(Z);
  \item $\alpha_n > 0$ for every $n\in \N$ 
	and $\sum_{n=1}^\infty \alpha_n = \infty$. 
 \end{itemize}
 Then $\{x_n\}$ converges strongly to $w=Q_F(u)$. 
\end{theorem}

First, we show some lemmas; then we prove Theorem~\ref{th:tauSRNS}. 
In the rest of this section, we set
\[
 y_n = J^{-1}\bigl( \alpha_n J u + (1-\alpha_n)J S_n x_n \bigr)
\]
for $n \in \N$, so \eqref{eq:def1} is reduced to $x_{n+1}= Q_C (y_n)$. 

\begin{lemma}\label{lm:bdd}
 Both $\{x_n\}$ and $\{S_n x_n\}$ are bounded,
 and moreover, the following hold: 
 \begin{enumerate}
  \item $y_n - S_n x_n \to 0$; \label{y_n-S_nx_n-to0}
  \item $\phi(w,x_{n+1}) \leq \alpha_n \phi(w,u) + \phi(w,S_nx_n)$
	for every $n\in \N$; \label{basic1}
  \item $\phi(w,x_{n+1}) \leq (1-\alpha_n)\phi(w,x_n) + 
	2 \alpha_n \ip{y_n - w}{Ju-Jw}$
	for every $n\in \N$. \label{basic2}
 \end{enumerate}
\end{lemma}

\begin{proof}
 Since $Q_C$ and $S_n$ are of type~(r) and $w\in F(S_n) \subset C$, 
 it follows from~\eqref{eq:J-convex} that
 \begin{align}
  \begin{split}\label{eq:bdd}
  \phi(w,x_{n+1}) 
   &\leq \phi(w,y_n)\\
  &\leq \alpha_n \phi(w,u) + (1-\alpha_n)\phi(w,S_n x_n)\\
  &\leq \alpha_n \phi(w,u) + (1-\alpha_n)\phi(w,x_n)
  \end{split}
 \end{align}
 for every $n\in \N$. 
 Thus, by induction on $n$, we have
 \[
  \phi(w,S_n x_n) \leq \phi(w,x_n) 
 \leq \max \{ \phi(w,x_1), \phi(w,u)\}. 
 \]
 Therefore, by virtue of~\eqref{eqn:phi-lower}, 
 it turns out that $\{x_n\}$ and $\{S_n x_n\}$ are bounded. 

 By $\alpha_n \to 0$, it is clear that 
 $Jy_n - JS_n x_n = \alpha_n (Ju - JS_n x_n)\to 0$. 
 This shows that 
 \[
 y_n - S_n x_n = J^{-1}Jy_n - J^{-1}JS_n x_n \to 0
 \]
 because $E^*$ is uniformly smooth and 
 $J^{-1}$ is uniformly continuous on every bounded set.
 Thus \eqref{y_n-S_nx_n-to0} holds. 

 \eqref{basic1} follows from~\eqref{eq:bdd}. 

 Since $S_n$ is of type~(r), it follows from \eqref{eq:bdd}, 
 \eqref{eq:KT2004}, and~\eqref{eq:J-convex} that 
 \begin{align}
  \begin{split}\label{eq:final}
   \phi(w,x_{n+1}) &\leq \phi(w,y_n)\\
   & \leq \phi \Bigl( w, 
   J^{-1} \bigl( 
   \alpha_n Ju + (1-\alpha_n) J S_n x_n  - \alpha_n (Ju-Jw)
   \bigr) \Bigr) \\
   &\qquad + 2 \ip{y_n - w}{\alpha_n (Ju-Jw)}\\
   & \leq (1-\alpha_n)\phi(w,S_n x_n) + \alpha_n \phi(w,w) 
   + 2 \alpha_n \ip{y_n - w}{Ju-Jw}\\
   &\leq (1-\alpha_n)\phi(w,x_n) + 2 \alpha_n \ip{y_n - w}{Ju-Jw}
  \end{split}
 \end{align}
  for every $n\in \N$. 
  Therefore, \eqref{basic2} holds. 
\end{proof}

\begin{lemma}\label{th:srns}
 Suppose that
 \begin{equation}\label{eq:limsup_assumption}
  \limsup_{n \to \infty} \bigl(
   \phi(w, x_n) - \phi(w, x_{n+1}) \bigr) \leq 0.
 \end{equation}
 Then $\{x_n\}$ converges strongly to $w$. 
\end{lemma}

\begin{proof}
 We first show that $S_n x_n - x_n \to 0$. 
 Since $S_n$ is of type~(r), it follows from~\eqref{basic1} in
 Lemma~\ref{lm:bdd} that
 \[
  0 \leq \phi(w,x_n) - \phi(w,S_n x_n)
  \leq \phi(w,x_n) - \phi(w,x_{n+1}) + \alpha_n \phi(w,u)
 \]
 for every $n\in \N$, so $\phi(w,x_n) - \phi(w,S_n x_n)\to 0$ 
 by~\eqref{eq:limsup_assumption} and $\alpha_n \to 0$. 
 Since $\{S_n\}$ is a strongly relatively nonexpansive sequence
 and $\{x_n\}$ is bounded by Lemma~\ref{lm:bdd}, 
 $\phi(S_n x_n, x_n)\to 0$. Using Lemma~\ref{lm:kamimura2002},
 we conclude that $S_n x_n - x_n \to 0$. 

 We know that $y_n - S_n x_n \to 0$ by~\eqref{y_n-S_nx_n-to0} in 
 Lemma~\ref{lm:bdd} and $\{S_n\}$ satisfies the condition~(Z) 
 by assumption, so Lemma~\ref{lm:limsup_le_0} implies that 
 \[
 \limsup_{n\to \infty} \ip{y_n - w}{Ju-Jw}
 = \limsup_{n\to \infty} \ip{S_n x_n - w}{Ju-Jw} \leq 0. 
 \]
 It follows from~\eqref{basic2} in Lemma~\ref{lm:bdd} that
 \[
 \phi(w,x_{n+1}) 
 \leq (1-\alpha_n)\phi(w,x_n) + 2 \alpha_n \ip{y_n - w}{Ju-Jw}
 \]
 for every $n\in \N$. 
 Therefore, noting that $\sum_{n=1}^\infty \alpha_n = \infty$
 and using Lemma~\ref{lm:seq}, we conclude that $\phi(w,x_n) \to 0$, 
 and hence $x_n \to w$ by Lemma~\ref{lm:kamimura2002}.
\end{proof}

\begin{lemma}\label{lm:convergent}
 The real number sequence $\{ \phi (w, x_n) \}$ is convergent. 
\end{lemma}

\begin{proof}
 We assume, to obtain a contraction,  
 that $\{ \phi(w,x_n) \}$ is not convergent. 
 Then Lemma~\ref{lm:neo_mainge} implies that there exist $N\in \N$ and an
 eventually increasing function $\tau\colon \N \to \N$ such that
 \begin{equation}\label{eq:mainge1}
  \phi(w, x_{\tau(n)}) \le \phi(w, x_{\tau(n)+1})
 \end{equation}
 for every $n\in \N$ and 
 \begin{equation}\label{eq:mainge2}
  \phi(w,x_n) \le \phi(w, x_{\tau(n)+1})
 \end{equation}
 for every $n \geq N$. 

 We show that $S_{\tau(n)} x_{\tau(n)} - x_{\tau(n)} \to 0$.
 Since $S_{\tau(n)}$ is of type~(r), 
 it follows from~\eqref{eq:mainge1}, \eqref{basic1} in
 Lemma~\ref{lm:bdd}, and Lemma~\ref{2zero} that 
 \begin{align*}
  0 &\leq \phi(w,x_{\tau(n)}) - \phi(w, S_{\tau(n)} x_{\tau(n)}) \\
  &\leq \phi(w, x_{\tau(n)+1}) - \phi(w, S_{\tau(n)} x_{\tau(n)}) \\
  &\leq \alpha_{\tau(n)} \phi(w,u) \to 0
 \end{align*}
 as $n\to \infty$. 
 Since $\{ x_{\tau(n)}\}$ is bounded and $\{ S_n \}$ is a strongly
 relatively nonexpansive sequence, it follows from Lemma~\ref{lm:tauSRNS}
 that $\phi( S_{\tau(n)} x_{\tau(n)}, x_{\tau(n)}) \to 0$, so we conclude
 that $S_{\tau(n)} x_{\tau(n)} -  x_{\tau(n)} \to 0$ by 
 Lemma~\ref{lm:kamimura2002}.

 Finally, we obtain a contradiction that $\phi(w,x_n)\to 0$. 
 From~\eqref{eq:mainge1} and~\eqref{basic2} in Lemma~\ref{eq:bdd}, 
 we know that
 \begin{align}
  \begin{split}\label{eq:x_tau_n+1}
   \phi(w,x_{\tau(n)}) &\leq \phi(w,x_{\tau(n)+1}) \\
  &\leq (1-\alpha_{\tau(n)}) \phi(w,x_{\tau(n)}) 
  + 2 \alpha_{\tau(n)} \ip{y_{\tau(n)} - w}{Ju-Jw}
  \end{split}
 \end{align}
 for every $n \in \N$, where $y_{\tau(n)} = J^{-1}\bigl( \alpha_{\tau(n)} J u 
 + (1-\alpha_{\tau(n)})J S_{\tau(n)} x_{\tau(n)} \bigr)$ for $n\in \N$. 
 Noting that $\alpha_{\tau(n)} > 0$, 
 \eqref{eq:x_tau_n+1} is reduced to 
 \begin{equation*}
 \phi(w, x_{\tau(n)}) \leq 2 \ip{y_{\tau(n)} - w}{Ju-Jw},
 \end{equation*}
 so that 
 \begin{equation}\label{eq:phi_leq_2}
  \phi(w,x_{\tau(n)+1}) \leq 2 \ip{y_{\tau(n)} - w}{Ju-Jw}
 \end{equation}
 for every $n \in \N$. 
 Since $\{ S_n \}$ satisfies the condition~(Z), it follows from
 Lemma~\ref{lm:tauZ} that every weak cluster point of $\{ x_{\tau(n)}\}$
 belongs to $F$. 
 Using~\eqref{eq:phi_leq_2}, \eqref{y_n-S_nx_n-to0} in Lemma~\ref{lm:bdd}, 
 and Lemma~\ref{lm:limsup_le_0}, we have 
\begin{align*}
  \begin{split}
   \limsup_{n\to \infty} \phi(w,x_{\tau(n)+1}) 
   &\leq 2 \limsup_{n\to \infty} \ip{y_{\tau(n)} - w}{Ju-Jw}\\
   & = 2 \limsup_{n\to \infty} 
   \ip{S_{\tau(n)} x_{\tau(n)} - w}{Ju-Jw} \leq 0.
  \end{split}
 \end{align*}
 Therefore, by virtue of~\eqref{eq:mainge2}, we conclude that
 \begin{align*}
  \limsup_{n\to \infty} \phi(w,x_n) 
  &\leq \limsup_{n\to \infty} \phi(w,x_{\tau(n)+1}) \leq 0,
 \end{align*}
 and hence $\phi(w,x_n)\to 0$, which is a contradiction. 
\end{proof}

\begin{proof}[Proof of Theorem~\ref{th:tauSRNS}]
 Using Lemmas~\ref{th:srns} and~\ref{lm:convergent}, we get the conclusion. 
\end{proof}

\section{Applications}\label{s:app}
In this section, we study the zero point problem for a maximal monotone
operator and the fixed point problem for a relatively nonexpansive
mapping. 
We employ Theorem~\ref{th:tauSRNS} to 
approximate solutions of these problems. 

To prove the first theorem, we need the following lemma: 

\begin{lemma}[\cite{TMJ}*{Lemma~3.5}] \label{lm:L_r-Z}
 Let $E$ be a strictly convex and reflexive Banach space
 whose norm is uniformly G\^{a}teaux differentiable,
 $\{r_n\}$ a sequence of positive real numbers, and $L_{r_n}$ the
 resolvent of a maximal monotone operator $A\subset E\times E^*$. 
 Suppose that $\inf_n r_n >0$ and $A^{-1}0$ is nonempty. 
 Then $\{L_{r_n}\}$ satisfies the condition~(Z). 
\end{lemma}

We adopt a modified proximal point algorithm introduced by Kohsaka and
Takahashi~\cite{MR2058504} in the following theorem: 

\begin{theorem}\label{th:mm}
 Let $E$ be a uniformly convex Banach space 
 whose norm is uniformly G\^{a}teaux differentiable, 
 $A \subset E\times E^*$ a maximal monotone operator,
 $\{\alpha_n \}$ a sequence in $(0,1]$, 
 and $\{r_n\}$ a sequence of positive real numbers. 
 Suppose that $A^{-1}0$ is nonempty, 
 $\alpha_n \to 0$, $\sum_{n=1}^\infty \alpha_n = \infty $, 
 and $\inf_n r_n > 0$. 
 Let $u$ be a point in $E$ and $\{x_n\}$ a sequence defined by 
 $x_1 \in C$ and 
 \begin{equation}\label{eq:def2-x_n}
  x_{n+1} = J^{-1} \bigl(\alpha_n J u 
   + (1 - \alpha_n) J L_{r_n} x_n \bigr)
 \end{equation}
 for $n\in \N$, where $L_{r_n} = (J+ r_n A)^{-1}J$. 
 Then $\{ x_n \}$ converges strongly to $Q_{A^{-1}0}(u)$.
\end{theorem}

\begin{proof}
 Set $S_n = L_{r_n}$ for $n\in \N$. 
 It is known that $F(S_n) = A^{-1}0$ and $L_{r_n}$ is a type~(r)
 self-mapping of $E$ for each $n\in\N$. 
 Hence $\bigcap_{n=1}^\infty F(S_n) = A^{-1} 0$ is nonempty.
 It is also known that $\{ S_n \}$ is a strongly relatively 
 nonexpansive sequence by~\cite{MR2529497}*{Example 3.2}
 and $\{ S_n \}$ satisfies the condition~(Z) by Lemma~\ref{lm:L_r-Z}. 
 It is clear that $Q_E$ is the identity mapping on $E$. 
 Therefore, Theorem~\ref{th:tauSRNS} implies the conclusion. 
\end{proof}

\begin{remark}\label{r:ak_kt}
Theorem~\ref{th:mm} is similar to \cite{MR2058504}*{Theorem 3.3}.
 In~\cite{MR2058504}*{Theorem 3.3}, 
 $E$ is assumed to be smooth and uniformly convex and $\{\alpha_n \}$ in
 $[0,1]$ while $\{r_n\}$ is assumed to diverge to infinity. 
\end{remark}

To prove the next theorem, we need the following lemma: 

\begin{lemma}[\cite{MR2529497}*{Lemma 2.1}]\label{lemma:uc-ft}
 Let $\{x_n\}$ and $\{y_n\}$ be two bounded sequences in a uniformly
 convex Banach space $E$ and $\{\lambda_n\}$ a sequence in $[0,1]$ such
 that $\liminf_{n\to\infty}\lambda_n>0$. Suppose that 
 \[
 \lambda_n \norm{x_n}^2 + (1-\lambda_n)\norm{y_n}^2 
 -\norm{\lambda_n x_n + (1-\lambda_n)y_n}^2 \to 0. 
 \]
 Then $(1-\lambda_n) (x_n-y_n) \to 0$. 
\end{lemma}

The following is a strong convergence theorem for a relatively
nonexpansive mapping; see~\cites{MR2058234,MR2142300} for 
other convergence theorems and see also~\cite{MR2529497}. 

\begin{theorem}[\cite{Nilsrakoo-Saejung}*{Theorem 3.4}]\label{th:fpp}
 Let $E$ be a uniformly convex and uniformly smooth Banach space, 
 $C$ a nonempty closed convex subset of $E$, 
 $T \colon C \to E$ a relatively nonexpansive mapping, 
 $\{\alpha_n \}$ a sequence in $(0,1]$, and
 $\{\beta_n \}$ a sequence in $[0,1]$. 
 Suppose that $\alpha_n \to 0$, $\sum_{n=1}^\infty \alpha_n = \infty $, 
 and $0 < \liminf_{n\to \infty} \beta_n 
 \leq \limsup_{n\to \infty}\beta_n <1$. 
 Let $u$ be a point in $E$ and $\{x_n\}$ a sequence defined by 
 $x_1 \in C$ and 
 \begin{equation}\label{def:x_n3}
  x_{n+1} = Q_C J^{-1} \Bigl(\alpha_n J u + (1 - \alpha_n)
   \bigl( \beta_n J x_n + (1-\beta_n)JT x_n \bigl) \Bigr)
 \end{equation}
 for $n\in \N$. 
 Then $\{ x_n \}$ converges strongly to $Q_{F(T)}(u)$.
\end{theorem}

\begin{proof}
 Set $S_n = J^{-1}\bigl( \beta_n J + (1-\beta_n)JT \bigr)$ 
 for $n\in \N$. Then it is easy to check that each $S_n$ is a mapping of
 type~(r) and $\bigcap_{n=1}^\infty F(S_n) =F(T)$; 
 see~\cite{AKT1}*{Corollary 3.8}. 
 Moreover, it is clear that 
 $\eqref{def:x_n3}$ coincides with $\eqref{eq:def1}$. 
 To finish the proof, it is enough to show that $\{S_n\}$ is a
 strongly relatively nonexpansive sequence 
 and $\{S_n\}$ satisfies the condition~(Z). 

 Let $\{y_n\}$ be a bounded sequence in $C$ such that
 $\phi(p, y_n) - \phi(p, S_{n} y_n) \to 0$ 
 for some $p\in \bigcap_{n=1}^\infty F(S_{n})$. 
 Since $T$ is of type~(r), we have
 \begin{align*}
  &\beta_{n} \norm{Jy_n}^2 + (1-\beta_{n})
  \norm{JT y_n}^2 - \norm{J S_{n} y_n}^2\\
  &\qquad = \beta_{n} \phi(p, y_n) + (1-\beta_{n}) 
  \phi(p, T y_n) - \phi(p, S_{n}y_n)\\
  &\qquad \leq 
  \beta_{n} \phi(p, y_n) + (1-\beta_{n}) \phi(p,y_n) 
  - \phi(p, S_{n} y_n)\\
  &\qquad = 
  \phi(p, y_n) -  \phi(p, S_{n} y_n) \to 0. 
 \end{align*} 
 Using Lemma~\ref{lemma:uc-ft} and~\eqref{eq:equiv_x_n-y_n}, 
 it turns out that 
 \begin{equation*}
  Jy_n - JS_{n} y_n = (1-\beta_{n} )(J y_n  - JT y_n )\to 0,
 \end{equation*}
 and hence $\phi(S_{n}y_n,y_n)\to 0$. 
 Thus $\{S_{n}\}$ is a strongly relatively nonexpansive sequence. 

 Let $\{z_n\}$ be a bounded sequence in $C$ such that 
 $z_n - S_{n} z_n \to 0$. 
 Then it follows from~\eqref{eq:equiv_x_n-y_n} that
 \[
 (1-\beta_{n})(J z_n - JT z_n ) = Jz_n - JS_{n} z_n \to 0, 
 \]
 so we conclude that $z_n - T z_n \to 0$ 
 by $\limsup_{n\to \infty}\beta_{n} <1$ and \eqref{eq:equiv_x_n-y_n}. 
 Since $T$ is relatively nonexpansive, 
 every weak cluster point of $\{z_n \}$ belongs to $F(T)$. 
 This means that $\{S_n\}$ satisfies the condition~(Z). 
 Consequently, Theorem~\ref{th:tauSRNS} implies the conclusion. 
\end{proof}

\begin{remark}
 In \cite{Nilsrakoo-Saejung}*{Theorem 3.4}, 
 $\{\alpha_n\}$ and $\{\beta_n\}$ are assumed to be sequences in
 $(0,1)$. 
\end{remark}


\begin{bibdiv}
 \begin{biblist}
\bib{MR1386667}{article}{
   author={Alber, Y. I.},
   title={Metric and generalized projection operators in Banach spaces:
   properties and applications},
   conference={
      title={Theory and applications of nonlinear operators of accretive and
      monotone type},
   },
   book={
      series={Lecture Notes in Pure and Appl. Math.},
      volume={178},
      publisher={Dekker},
      place={New York},
   },
   date={1996},
   pages={15--50},
}
%

\bib{MR2338104}{article}{
   author={Aoyama, Koji},
   author={Kimura, Yasunori},
   author={Takahashi, Wataru},
   author={Toyoda, Masashi},
   title={Approximation of common fixed points of a countable family of
   nonexpansive mappings in a Banach space},
   journal={Nonlinear Anal.},
   volume={67},
   date={2007},
   pages={2350--2360},
}

\bib{MR2529497}{article}{
   author={Aoyama, Koji},
   author={Kohsaka, Fumiaki},
   author={Takahashi, Wataru},
   title={Strongly relatively nonexpansive sequences in Banach spaces and
   applications},
   journal={J. Fixed Point Theory Appl.},
   volume={5},
   date={2009},
   pages={201--224},
}

\bib{TMJ}{article}{
   author={Aoyama, Koji},
   author={Kohsaka, Fumiaki},
   author={Takahashi, Wataru},
   title={Proximal point methods for monotone operators in
	   Banach spaces},
 journal={Taiwanese Journal of Mathematics}, 
   volume={15},
   date={2011},
   pages={259--281},
}

\bib{AKT1}{article}{
  author={Aoyama, Koji},
  author={Kohsaka, Fumiaki},
  author={Takahashi, Wataru},
  title={Strong convergence theorems by shrinking and hybrid projection 
  methods for relatively nonexpansive mappings in Banach spaces},
  conference={
  title={Nonlinear analysis and convex analysis},
  },
  book={
  publisher={Yokohama Publ., Yokohama},
  },
  date={2009},
  pages={7--26},
}

\bib{MR1402641}{article}{
   author={Censor, Y.},
   author={Reich, S.},
   title={Iterations of paracontractions and firmly nonexpansive operators
   with applications to feasibility and optimization},
   journal={Optimization},
   volume={37},
   date={1996},
   pages={323--339},
}


\bib{MR2112848}{article}{
   author={Kamimura, Shoji},
   author={Kohsaka, Fumiaki},
   author={Takahashi, Wataru},
   title={Weak and strong convergence theorems for maximal monotone
   operators in a Banach space},
   journal={Set-Valued Anal.},
   volume={12},
   date={2004},
   pages={417--429},
}

\bib{MR1972223}{article}{
   author={Kamimura, Shoji},
   author={Takahashi, Wataru},
   title={Strong convergence of a proximal-type algorithm in a Banach space},
   journal={SIAM J. Optim.},
   volume={13},
   date={2002},
   pages={938--945 (electronic) (2003)},
}


\bib{MR2058504}{article}{
   author={Kohsaka, Fumiaki},
   author={Takahashi, Wataru},
   title={Strong convergence of an iterative sequence for maximal monotone
   operators in a Banach space},
   journal={Abstr. Appl. Anal.},
   date={2004},
   pages={239--249},
}

\bib{MR2466027}{article}{
   author={Maing{\'e}, Paul-Emile},
   title={Strong convergence of projected subgradient methods for nonsmooth
   and nonstrictly convex minimization},
   journal={Set-Valued Anal.},
   volume={16},
   date={2008},
   pages={899--912},
}

\bib{MR2058234}{article}{
   author={Matsushita, S.},
   author={Takahashi, Wataru},
   title={Weak and strong convergence theorems for relatively nonexpansive
   mappings in Banach spaces},
   journal={Fixed Point Theory Appl.},
   date={2004},
   pages={37--47},
}

\bib{MR2142300}{article}{
   author={Matsushita, S.},
   author={Takahashi, Wataru},
   title={A strong convergence theorem for relatively nonexpansive mappings
   in a Banach space},
   journal={J. Approx. Theory},
   volume={134},
   date={2005},
   pages={257--266},
}

\bib{Nilsrakoo-Saejung}{article}{
   author={Nilsrakoo, Weerayuth},
   author={Saejung, Satit},
   title={Strong convergence theorems by Halpern-Mann iterations for
   relatively nonexpansive mappings in Banach spaces},
   journal={Applied Mathematics and Computation},
   volume={217},
   date={2011},
   pages={6577--6586},
}

\bib{MR1386686}{article}{
   author={Reich, Simeon},
   title={A weak convergence theorem for the alternating method with Bregman
   distances},
   conference={
      title={Theory and applications of nonlinear operators of accretive and
      monotone type},
   },
   book={
      series={Lecture Notes in Pure and Appl. Math.},
      volume={178},
      publisher={Dekker},
      place={New York},
   },
   date={1996},
   pages={313--318},
}

\bib{MR0282272}{article}{
   author={Rockafellar, R. T.},
   title={On the maximality of sums of nonlinear monotone operators},
   journal={Trans. Amer. Math. Soc.},
   volume={149},
   date={1970},
   pages={75--88},
}

\bib{MR2680036}{article}{
   author={Saejung, Satit},
   title={Halpern's iteration in Banach spaces},
   journal={Nonlinear Anal.},
   volume={73},
   date={2010},
   pages={3431--3439},
}

\bib{MR1864294}{book}{
   author={Takahashi, Wataru},
   title={Nonlinear functional analysis},
   publisher={Yokohama Publishers, Yokohama},
   date={2000},
   pages={iv+276},
}

\bib{MR1911872}{article}{
   author={Xu, Hong-Kun},
   title={Iterative algorithms for nonlinear operators},
   journal={J. London Math. Soc. (2)},
   volume={66},
   date={2002},
   pages={240--256},
}
\end{biblist}
\end{bibdiv}
\end{document}